\documentclass{amsart}
\usepackage{amsfonts}
\usepackage{bbm}
\usepackage{amsmath}
\usepackage{bigdelim}
\usepackage{multirow}
\usepackage{amssymb}
\usepackage{indentfirst}
\usepackage[dvips]{graphicx}
\usepackage{fancyhdr}
\usepackage{amscd,amssymb,amsmath,graphicx,verbatim}
\usepackage[TS1,OT1,T1]{fontenc}

\newtheorem{theorem}{Theorem}[section]
\newtheorem{lemma}[theorem]{Lemma}

\newtheorem{proposition}[theorem]{Proposition}

\theoremstyle{definition}

\newtheorem{example}[theorem]{Example}

\newtheorem{question}[theorem]{Question}
\theoremstyle{remark}

\numberwithin{equation}{section}

\begin{document}
\title{Angles and Schauder basis in Hilbert spaces}
\author{Bingzhe Hou }
\address{Bingzhe Hou, School of Mathematics, Jilin university, 130012, Changchun, P.R.China} \email{houbz@jlu.edu.cn}
\author{Yang Cao }
\address{Yang Cao, School of Mathematics, Jilin university, 130012, Changchun, P.R.China} \email{caoyang@jlu.edu.cn}
\author{Geng Tian}
\address{Geng Tian, Department of Mathematics, Liaoning University, 110036, Shenyang, P. R. China} \email{tiangeng09@mails.jlu.edu.cn}
\author{Xinzhi Zhang}
\address{Xinzhi Zhang, Liaoning Province Shiyan High School, 110036, Shenyang, P. R. China} \email{xzzhangmath@sina.com}
\date{}
\subjclass[2000]{Primary 46B15; Secondary 46B20,47B37.}
\keywords{Schauder basis, angles, Hilbert spaces, measures.}
\begin{abstract}
Let $\mathcal{H}$ be a complex separable Hilbert space. We prove
that if $\{f_{n}\}_{n=1}^{\infty}$ is a Schauder basis of the
Hilbert space $\mathcal{H}$, then the angles between any two vectors
in this basis must have a positive lower bound. Furthermore, we
investigate that $\{z^{\sigma^{-1}(n)}\}_{n=1}^{\infty}$ can never
be a Schauder basis of $L^{2}(\mathbb{T},\nu)$, where $\mathbb{T}$ is the
unit circle, $\nu$ is a finite positive discrete measure, and $\sigma: \mathbb{Z}
\rightarrow \mathbb{N}$ is an arbitrary surjective and injective
map.
\end{abstract}
\maketitle

\section{Introduction and preliminaries}
Let $\mathcal{H}$ be a complex separable Hilbert space. In this
paper, we are interested in Schauder basis in $\mathcal{H}$. Recall
that a sequence
$\psi=\{f_{n}\}_{n=1}^{\infty}$ is called a minimal sequence if $f_{k} \notin span\{f_{n}; n \in \mathbb{N}, n \ne k\}$.
A sequence $\psi=\{f_{n}\}_{n=1}^{\infty}$ is called a Schauder
basis of the Hilbert space $\mathcal{H}$ if for every
vector $x \in \mathcal{H}$ there exists a unique sequence
$\{\alpha_{n}\}_{n=1}^{\infty}$ of complex numbers such that the
partial sum sequence $x_{k}=\sum_{n=1}^{k} \alpha_{n}f_{n}$
converges to $x$ in norm.

Suppose that $\{f_{n}\}_{n=1}^{\infty}$ is a Schauder basis of the
Hilbert space $\mathcal{H}$. As it is well-known, we can define the
angle between any two vectors in this basis by
$$
\theta_{kl}=\arccos \frac{|(f_{k}, f_{l})|}{||f_{k}||\cdot
||f_{l}||}.
$$
Then there is a natural question as follows.
\begin{question}
Dose there exist some constant $\theta_{0}$ such that
$0<\theta_{0}<\theta_{kl}$ holds for every $k, l \in \mathbb{N}$?
\end{question}
We will give a positive answer to this question and furthermore,
discuss the Schauder basis of $L^2(\mathbb{T},\nu)$, where
$\mathbb{T}$ is denoted by the unit circle and $\nu$ is a finite
positive measure on $\mathbb{T}$. In this paper, we will use some descriptions of
Schauder basis in operator language, which were introduced in \cite{Cao}.

Let $\{e_{k}\}_{k=1}^{\infty}$ be an orthonormal basis (ONB in
brief) of $\mathcal{H}$. Denote by $P_{k}$ the the diagonal matrix
with the first $k-$th entries on diagonal valued 1 and others valued $0$. Then as an operator, $P_{k}$ represents the orthogonal
projection from $\mathcal{H}$ to the subspace
$\mathcal{H}^{(k)}=span \{e_{1}, e_{2}, \cdots, e_{k}\}$.

Moreover, for a Schauder basis $\{f_{n}\}_{n=1}^{\infty}$, denote
$$
F=\begin{matrix}\begin{bmatrix}
f_{11}&f_{12}&\cdots\\
f_{21}&f_{22}&\cdots\\
\vdots&\vdots&\ddots\\
\end{bmatrix}\begin{matrix}
\end{matrix}\end{matrix}
$$
where the entry $f_{nm}$ is the $n$-th coordinate of vector $f_m$
under the ONB $\{e_{k}\}_{k=1}^{\infty}$. We call $F$ the
Schauder matrix corresponding to the Schauder basis
$\{f_{n}\}_{n=1}^{\infty}$.

\begin{lemma}[\cite{Cao}]\label{Lemma:Matrix Form 2}
Assume that $\{f_{n}\}_{n=1}^{\infty}$ is a Schauder basis. Then the
corresponding Schauder matrix $F$ satisfies the following
properties:

1. Each column of the matrix $F$ is an $l^{2}-$ sequence;

2. $F$ has a unique left inverse $G^{*}=(g_{kl})$ such that
each row of $G^{*}$ is also an $l^{2}-$ sequence;

3. Each operator $Q_{k}$ $Q_{k}=FP_{k}G^{*}$ is a
well-defined projection on $\mathcal{H}$ and $Q_{k}$ converges to the unit
operator $I$ in the strong operator topology as $k$ tends to infinite.
\end{lemma}

Notice that $G^{*}$ does
not mean the adjoint of $G$, it is just a notation of left reverse, that
is, the series $\sum_{j=1}^{\infty}g_{kj}f_{jn}$ converges
absolutely to $\delta_{kn}$ for $k, n\in\mathbb{N}$.

The projection $FP_{n}G^{*}$ is the $n-$th ``natural projection'' so called in
\cite{B}(p354). It is also the $n-$th partial sum operator so called
in \cite{Singer}(definition 4.4, p25). Now we can translate theorem
4.1.15 and corollary 4.1.17 in \cite{B} into the following form.

\begin{proposition}[\cite{Cao}]\label{BC1}
If $F$ is a Schauder matrix, then $M=\sup_{n} \{||FP_{n}G^{*}||\}$
is a finite constant. Moreover, the constant $M$ is called the basis
constant for the Schauder basis $\{f_{n}\}_{n=1}^{\infty}$.
\end{proposition}

\section{Main results}


\begin{theorem}\label{Theorem: Angel between vectors in a basis}
Suppose that $\psi=\{f_{n}\}_{n=1}^{\infty}$ is a Schauder basis of
a Hilbert space $\mathcal{H}$ with basis constant $M$. Then the
angles between any two vectors in this basis must have a positive
lower bound. In fact,  $\theta_{kl}\geq\arccos(1-1/{8M^{2}})$
for all $k, l \in \mathbb{N}$.
\end{theorem}
\begin{proof}
Without losses, assume that the Schauder basis
$\psi$ is a normal basis, i.e., $||f_{n}||=1$ for all $n \in
\mathbb{N}$. Denote by  $\mathcal{M}_{kl}=\{\alpha f_{k}+\beta
f_{l}; \alpha, \beta \in \mathbb{C}\}$. Then $\mathcal{M}_{kl}$ is
just the plane spanned by $f_{k}$ and $f_{l}$. For any vector
$v=\alpha f_{k}+\beta f_{l} \in \mathcal{M}_{kl}$, one can uniquely write
$v=\sum_{n=1}^{\infty} \alpha_{n}f_{n}$, where $\alpha_{k}=\alpha,
\alpha_{l}=\beta$ and $\alpha_{n}=0$ for $n \ne k, l$. Moreover, the
natural projection $\widetilde{Q}_{k}=Q_{k}-Q_{k-1}$ satisfies
$\widetilde{Q}_{k}v=\alpha f_{k}$ and $||\widetilde{Q}_{k}|| \le 2M$
by proposition \ref{BC1}. By multiplying a unit complex number
$\lambda$ to the vector $f_{l}$ and using proposition 4.1.5 in the
book \cite{B} (p351), we also can assume $0 < (f_{k}, f_{l})$. Now
let $v_{kl}=(f_{k}-f_{l})/{||f_{k}-f_{l}||}$. If the conclusion dose
not hold, then for any positive number $0<r<1$ there exist some $k, l$
such that $r<(f_{k}, f_{l})<1$. Hence we have
$\widetilde{Q}_{k}v_{kl}={1}/{||f_{k}-f_{l}||}f_{k}$ and consequently
$$
||\widetilde{Q}_{k}|| \ge
||\widetilde{Q}_{k}v_{kl}||=\frac{1}{||f_{k}-f_{l}||} \ge
\frac{1}{\sqrt{2(1-r)}}
$$
which runs counter to the fact $||\widetilde{Q}_{k}|| \le 2M$ for
$r>1-{1}/{8M^{2}}$.
\end{proof}

\begin{example}
The theorem above holds for Schauder basis but not for minimal sequence. Let
$\mathcal{H}=\oplus_{n=1}^{\infty} \mathcal{H}_{n}$ where
$\mathcal{H}_{n}=\mathbb{C}^{2}$. We can pick a normal basis
$\{f^{(n)}_{1}, f^{(n)}_{2}\}$ for $\mathcal{H}_{n}$ such that
$(f^{(n)}_{1}, f^{(n)}_{2})> 1-{1}/{n}$. Let
$g_{2n-1}=f^{(n)}_{1}$ and $g_{2n}=f^{(n)}_{2}$ for $n\ge 1$. Then
it is easy to see that $\psi=\{g_{n}\}_{n=1}^{\infty}$ is a
minimal sequence but the
angles between any two vectors in this basis don't have a positive
lower bound.
\end{example}

Now we will use the theorem above to study the space of square integrable functions under a discrete measure.
Recall that a discrete measure is just a measure concentrated on a
countable set $E$ (\cite{Rudin}, p175).
We need a result in number theory which comes from the Farey
fractions.

\begin{lemma}\label{Theorem: Rational Approximations}
Let $x_{1}, x_{2}, \cdots, x_{k}$ be arbitrary real numbers, and $n$
be a positive integer. Then there exist integers $a_{1}, a_{2},
\cdots, a_{k}$ and an integer $b$, $0<b\le n$, such that
$$
|x_{i}-\frac{a_{i}}{b}| \le \frac{1}{bn^{\frac{1}{k}}}.
$$
\end{lemma}

This is  a conclusion on rational
approximation(\cite{Ivan-Hilbert-Hugh}, p316, theorem 6.25). We emphasize here that the integer $n$ is independent to the choice
of real numbers $x_{1}, x_{2}, \cdots, x_{k}$.

\begin{theorem}\label{Theorem: z Powers never Be a Basis of Discrete Measure}
Let $\nu$ be a finite positive discrete measure on the unit circle
$\mathbb{T}$. For any bijective map $\sigma: \mathbb{Z} \rightarrow
\mathbb{N}$, let $f_{n}=z^{\sigma^{-1}(n)}$. Then
$\{f_{n}\}_{n=1}^{\infty}$ can never be a basis of
$L^{2}(\mathbb{T},\nu)$.
\end{theorem}

\begin{proof}
Without losses, assume $\nu(\mathbb{T})=1$. Let $\{e^{i2\pi t_{j}}\}_{j=1}^{\infty}$ be the support set of the measure $\nu$. Given any positive number
$\epsilon <1$. One can choose a positive integer $k$ such that
$$
1 \le \sum_{j=1}^{k} \nu(e^{i2\pi t_{j}})+\epsilon,
$$
and a positive number $\delta$  such that
$$
1<\cos 2\pi x+\epsilon \ \ \hbox{ and } \ \  |\sin 2\pi x|<\epsilon, \ \  \hbox{ for } \
|x|<\delta.
$$
Let $n$ be a positive integer such that
${1}/{\delta}<n^{\frac{1}{k}}$. By lemma \ref{Theorem:
Rational Approximations}, there exist integers $a_{1}, a_{2},
\cdots, a_{k}$ and an integer $b$, $0<b\le n$, such that
$$
|t_{i}-\frac{a_{i}}{b}| \le
\frac{1}{bn^{\frac{1}{k}}}<\frac{\delta}{b}, \ \ \hbox{ for }1\le i \le
k.
$$
Then by the choice of $\delta$, we have
$$
1<\cos 2\pi bt_{i}+\epsilon \ \ and \ \  |\sin 2\pi bt_{i}|<\epsilon \ \ \hbox{ for
} 1\le i \le k.
$$
Hence
$$
|z^{b}(t_{i})-1|=|e^{i2\pi bt_{i}}-1|<2\epsilon, \ \ for 1=1, 2, \cdots, k.
$$
Consequently, we have
$$
\begin{array}{rl}
|\int_{0}^{1} z^{b} d\nu| &\ge |\sum_{i=1}^{k} (e^{i2\pi bt_{i}})\nu(e^{i2\pi t_{i}})|-\epsilon \\
                              &\ge (1-\epsilon)\sum_{i=1}^{k} \nu(e^{i2\pi t_{i}})-\epsilon\\
                              &\ge (1-\epsilon)^{2}-\epsilon \\
                              &>1-4\epsilon.
\end{array}
$$
Let $\epsilon_{n}={1}/{n}$. Then we can choose
$b_{n}$ such that $|\int_{0}^{1} z^{b_{n}} d\nu| \rightarrow
1$ as $n \rightarrow \infty$. Therefore, by theorem \ref{Theorem:
Angel between vectors in a basis}, we can finish the proof.
\end{proof}

\begin{theorem}
\label{Theorem: z Powers times f never Be a Basis of Discrete
Measure} Let $\nu$ be a finite positive discrete measure on the unit
circle $\mathbb{T}$. For any bijective map $\sigma: \mathbb{Z}
\rightarrow \mathbb{N}$ any $ f \in L^{2}(\nu)$, let
$f_{n}=z^{\sigma^{-1}(n)}f$. Then $\{f_{n}\}_{n=1}^{\infty}$ can
never be a Schauder basis of $L^{2}(\mathbb{T},\lambda_{a})$.
\end{theorem}
\begin{proof}
We will prove our theorem by Reductio ad absurdum. Suppose that there exists such $\{f_{n}\}_{n=1}^{\infty}$ that is a Schauder basis of $L^{2}(\mathbb{T},\lambda_{a})$.
Let $E=\{t_{k}\}_{k=1}^{\infty}$ be the support set of the discrete
measure $\nu$. Define a new finite discrete positive Borel measure
$\mu$ as follows:
$$
\mu(\{t_{k}\})=|f(t_{k})|^{2}\nu(\{t_{k}\}), \hbox{ for }k \in
\mathbb{N}.
$$
Define $V: L^{2}(\mu) \rightarrow L^{2}(\nu)$ by
$$
 Vg=gf,  \ \  for \ any \ g\in L^{2}(\mu).
$$
Obviously, $V$ is a
well-defined linear operator. Then $V$ is an isometry from the Hilbert space
$L^{2}(\mu)$ to the Hilbert space $L^{2}(\nu)$, since
$$
\begin{array}{rl}
||Vg||^{2}_{L^{2}(\nu)} &=\sum_{k=1}^{\infty} |gf(t_{k})|^{2}\nu(t_{k})\\
&=\sum_{k=1}^{\infty} |g(t_{k})|^{2}|f(t_{k})|^{2}\nu(t_{k})\\
&=\sum_{k=1}^{\infty} |g(t_{k})|^{2}\mu(t_{k})\\
&=||g||^{2}_{L^{2}(\mu)}.
\end{array}
$$
Moreover, $V$ must be surjective since $\{f_{n}\}_{n=1}^{\infty}$ is
a Schauder basis. Hence $V$ is an isometric isomorphism. Let
$g_{n}=z^{\sigma^{-1}(n)}=V^{-1}f_{n}$. Then by the Arsove's
theorem, $\{g_{n}\}_{n=1}^{\infty}$ must be a Schauder basis of $L^{2}(\mu)$, which runs counter to theorem \ref{Theorem: z Powers never Be a
Basis of Discrete Measure}.
\end{proof}

\section{Some remarks}

In this section, we will study the relations between diagonal operators and shift operators. Let us begin with an observation on finite
dimensional Hilbert space. Let $\psi=\{f_{1}, \cdots, f_{n}\}$ be a Schauder
basis of an $n$-dimensional Hilbert space $\mathcal{H}$ and $\pi$ be the permutation of $n$ defined by
$$
\begin{array}{rl}
\pi =\left(
      \begin{array}{ccccc}
       1 & 2 & 3 & \cdots & n \\
       2 & 3 & 4 & \cdots & 1
      \end{array}
     \right)
\end{array}.
$$
Then the operator $S_{\psi}$ defined by $S_{\psi}(f_{k})=f_{\pi(k)}$, can be seemed as a bilateral shift on basis $\psi$.
However, $S_{\psi}$ is diagonalizable, since it has $n$ different eigenvalues
$e^{i\frac{2k\pi}{n}}, k=1, 2, \cdots, n$.

Now we turn to consider the multiplication operator $M_{z}$ on
$L^{2}(\mathbb{T},\nu)$ in which $\nu$ is a finite positive Borel
measure on $\mathbb{T}$. As well known, if $\nu$ is the Lebesgue measure
$m$, then $M_{z}$ is the bilateral shift operator. While $\nu$ is a
finite positive discrete Borel measure on $\mathbb{T}$, then $M_{z}$
is a diagonal operator. One could consider the following problem.
\begin{question}
Let $\nu$ be a finite positive discrete measure on $\mathbb{T}$. On
the Hilbert space $L^{2}(\mathbb{T},\nu)$, can we choose a Schauder
basis $\psi=\{\psi_{n}\}_{n=1}^{\infty}$ such that $M_{z}$ can be
written as a shift operator on $\psi$?
\end{question}

The answer is positive if $Card\{z\in\mathbb{T};\nu(\{z\})>0\}$ is
finite and is negative if $Card\{z\in\mathbb{T};\nu(\{z\})>0\}$ is
infinite.

\begin{theorem}\label{Theorem: A Multi-Op be a shift on basis}
Let $\nu$ be a finite positive discrete measure on the unit circle
$\mathbb{T}$ with finite $Card\{z\in\mathbb{T};\nu(\{z\})>0\}$, and
let $M_{z}$ be the corresponding multiplication operator on
$L^{2}(\mathbb{T},\nu)$. Then $M_{z}$ can be written as a shift on
some Schauder basis $\psi$ of $L^{2}(\mathbb{T},\nu)$.
\end{theorem}
\begin{proof}
Suppose $Card\{z\in\mathbb{T};\nu(\{z\})>0\}=n<\infty$. Then
$L^{2}(\mathbb{T},\nu)$ is an $n$-dimensional Hilbert space and $M_{z}$
is a diagonal operator on it with $n$ different eigenvalues. Thus
$M_{z}$ is cyclic and consequently one can select a cyclic vector $f$.
Therefore, $\psi=\{M_{z}^kf\}_{k=0}^{n-1}$ is a Schauder basis of
$L^{2}(\mathbb{T},\nu)$ and $M_{z}$ can be written as a shift on
$\psi$.
\end{proof}

\begin{theorem}\label{Theorem: A Multi-Op never be a shift on basis}
Let $\nu$ be a finite positive discrete measure on the unit circle
$\mathbb{T}$ with infinite $Card\{z\in\mathbb{T};\nu(\{z\})>0\}$,
and let $M_{z}$ be the corresponding multiplication operator on
$L^{2}(\mathbb{T},\nu)$. Then $M_{z}$ can never be written as a
shift on a Schauder basis $\psi$ of $L^{2}(\mathbb{T},\nu)$.
\end{theorem}
\begin{proof}
Suppose that there exist a bijection $\sigma : \mathbb{N}
\rightarrow \mathbb{Z}$ and a Schauder basis
$\psi=\{\psi_{n}\}_{n=1}^{\infty}$ such that
$M_{z}\psi_{\sigma^{-1}(m)}=\psi_{\sigma^{-1}(m+1)}$. Then
$\psi_{\sigma^{-1}(m)}=z^{m}\psi_{\sigma^{-1}(0)}$. Let $f=\psi_{\sigma^{-1}(0)}$, then
$\psi_{k}=z^{\sigma(k)}f$. Thus, the sequence
$\{z^{\sigma(k-1)}f\}_{k=1}^{\infty}$
is a basis of $L^{2}(\mathbb{T},\nu)$, which contradicts with theorem \ref{Theorem: z Powers
times f never Be a Basis of Discrete Measure}.
\end{proof}


\begin{thebibliography}{000}
\bibitem{Cao} Cao Y, Tian G, Hou B. Schauder Bases and Operator
Theory. Available at http://arxiv.org/abs/1203.3603.

\bibitem{B} Megginson R. An introuduction to Banach Space Theory. New York, NY, USA: Springer-Verlag, 1998.

\bibitem{Ivan-Hilbert-Hugh} Niven I, Zuckerman H, Montgomery H.
An introduction to the theory of numbers. New York, NY, USA: John Wiley
and Sons, Inc., 1991.

\bibitem{Rudin}
Rudin W. Real and complex analysis. New York, NY, USA: McGraw-Hill Book
Co., 1987.


\bibitem{Singer} Singer I.Bases in Banach Space I. New York, NY, USA: Springer-verlag, 1970.

\end{thebibliography}
\end{document}